\documentclass[11pt]{article}
\usepackage{amsmath, amsthm, amssymb, graphicx,caption,subcaption}
\usepackage{makeidx}  %%% standard INDEX
\makeindex

\newtheorem{theorem}{Theorem}
\newtheorem*{theorem*}{Theorem}
\newtheorem{corollary}[theorem]{Corollary}
\newtheorem{prop}[theorem]{Proposition}
\newtheorem{lem}[theorem]{Lemma}
\newtheorem{conjecture}{Conjecture}

\newtheorem*{wilson}{Wilson's Theorem}

\theoremstyle{definition}

\newtheorem{definition}[theorem]{Definition}
\newtheorem{rem}[theorem]{Remark}

\def\A{\mathcal{A}}
\def\B{\mathcal{B}}
\def\C{\mathcal{C}}

\allowdisplaybreaks[1]

\begin{document}

\title{Period Preserving Properties of an Invariant from the Permanent of Signed Incidence Matrices}
\author{Iain Crump, Matt DeVos, Karen Yeats}
\date{}
\maketitle

\normalsize
\vspace{1cm}

\begin{abstract}
A 4-point Feynman diagram in scalar $\phi^4$ theory is represented by a graph $G$ which is obtained from a connected 4-regular graph by deleting a vertex.  The associated Feynman integral gives a quantity called the period of $G$ which is invariant under a number of graph operations --- namely, planar duality, the Schnetz twist, and it also does not depend on the choice of vertex deleted to form $G$.

In this article we study a graph invariant we call the graph permanent, which was implicitly introduced in a paper by Alon, Linial and Meshulam \cite{AlLiMe}.  The graph permanent applies to any graph $G = (V,E)$ for which $|E|$ is a multiple of $|V| - 1$ (so in particular to graphs obtained from a 4-regular graph by removing a vertex).  We prove that the graph permanent, like the period, is invariant under planar duality and the Schnetz twist when these are valid operations, and we show that when $G$ is obtained from a $2k$-regular graph by deleting a vertex, the graph permanent does not depend on the choice of deleted vertex.  
\end{abstract}

\section{Introduction}

For the duration of this introduction, let $\Gamma$ be a 4-regular graph and let $G=\Gamma - v$ for some $v\in V(\Gamma)$.  The graph $\Gamma$ can be uniquely reconstructed from~$G$; we call $\Gamma$ the \emph{completion} of $G$ and $G$ a \emph{decompletion} of $\Gamma$.  We can think of~$G$ as a Feynman diagram, specifically as a 4-point graph in scalar $\phi^4$ theory.  It is natural, then, to ask about the Feynman integral of $G$.  Simplifying by ignoring all physical parameters and considering only graphs with no subdivergences, it makes sense to define a number known as the \emph{period} of~$G$ \cite{bek, Brbig, Mar, Sphi4} (defined more thoroughly in Section~\ref{relationtofeynman}).  This period is the residue of the Feynman integral under a variety of regularizations and so gives a largely renormalization scheme independent aspect of the Feynman integral.

There are a number of graph theoretic operations which are known to preserve the period.  If $G$ is planar, then $G$ and its planar dual have the same period \cite{bkphi4}; this is a consequence of taking a Fourier transform of the Feynman integral.  If $G$ and another graph $G'$ have isomorphic completions, then $G$ and $G'$ have the same period \cite{bkphi4, Sphi4}.  If $\Gamma$ and $\Gamma'$ relate by the Schnetz twist (see Figure~\ref{twist}) and $G$ and $G'$ are decompletions of $\Gamma$ and $\Gamma'$ respectively, then $G$ and $G'$ have the same period \cite{Sphi4}.  Furthermore, if $\Gamma$ has a 3-vertex cut, then the period of $G$ is a product of the periods of two particular minors \cite{bkphi4, Sphi4}.

In view of this, we are interested in graph theoretic properties or invariants which are preserved by planar duality, completion followed by decompletion, and Schnetz twist.  The $c_2$ invariant is an arithmetic graph invariant defined by counting points on the Kirchhoff polynomial which is conjectured to have these properties (see \cite{BrS}); duality is proven in \cite{Dor} for graphs that meet a specific subgraph condition, and it is established in \cite{BrS} that if decompletion is true then the 3-cut condition follows.  We are not aware of any other nontrivial graph invariants thought to satisfy these properties.  

In this paper we introduce the following graph invariant; let $G$ be a graph with $|E(G)| = k(|V(G)|-1)$ for some integer $k$. Construct a signed incidence matrix from $G$ and delete a single arbitrary row. From this, construct a block matrix by stacking the modified incidence matrix $k$ times. Up to sign, we call the permanent of this matrix modulo $k+1$ the \emph{graph permanent}. Extending the concept of completion and decompletion to arbitrary regular graphs (Definition \ref{decompdef}), we prove the following:

\begin{theorem*} Suppose $\Gamma$ and $\Gamma '$ are connected $2k$-regular graphs. 
\begin{itemize}
\item Any two decompletions of $\Gamma$ have equal graph permanent (Theorem~\ref{invariance}).
\item If $\Gamma$ and $\Gamma'$ differ by a Schnetz twist, any pair of decompletions of $\Gamma$ and $\Gamma'$ will have equal graph permanents  (Proposition~\ref{schnetz}).  
\end{itemize} 
Further, let $G$ be a graph such that $|E(G)| = 2(|V(G)|-1)$ and $G^*$ its planar dual.
\begin{itemize}
\item The graph permanents of $G$ and $G^*$ are equal (Proposition~\ref{dual}).
\end{itemize}
Finally, suppose $\Gamma$ is a $4$-regular graph.
\begin{itemize}
\item If $\Gamma$ has a $3$-vertex cut, then the graph permanent of any decompletion of $\Gamma$ is the product of the graph permanents of two particular minors (Corollary~\ref{3cut}).
\end{itemize} \end{theorem*}

\noindent Further we prove a product property when $K$ has a 4-edge cut (Theorem~\ref{4edgecut}) that corresponds to the case of subdivergences in the Feynman graph.

 Our invariant is well defined on a wider class of graphs than decompletions of 4-regular graphs and relates naturally to flow questions on graphs.  Indeed, when this invariant is nonzero, it implies, by way of the Alon-Tarsi polynomial technique (\cite{AlTar}), that the graph in question has a modular $k$ orientation (or equivalently a $\mathbb{Z}_k$ flow using only the values $\pm 1$).  Furthermore, in the key case of 4-regular graphs, the polynomial is closely related to the Feynman integrand.  This will be explained in detail in Section~\ref{flow_sec}

Completion invariance for $2k$-regular graphs can be distilled into a curious identity for graphs (Theorem~\ref{agreementidentity}), and we close the introduction with a  description of this.   Let $G$ be $2k$-regular, and fix an orientation of the edges of $G$ which we call the \emph{reference orientation}.  Now define an arbitrary orientation of $G$ to be \emph{odd} \emph{(even)} if the number of edges for which this orientation disagrees with the reference orientation is odd (even).  Let $s,t$ be distinct vertices of $G$.  Any orientation for which $\mathrm{deg}^+(s) = 2k = \mathrm{deg}^-(t)$ and $\mathrm{deg}^+(v) = k = \mathrm{deg}^-(v)$ for every $v \in V(G) \setminus \{s,t\}$ will be called an $s$-\emph{to}-$t$ orientation.  Let $E_{s, t}$ $(O_{s,  t})$ denote the number of even (odd) $s$-to-$t$ orientations.  With this terminology, we can state this new identity as follows.

\begin{theorem*}
If $G$ is a $2k$-regular graph, and $(s,t)$, $(s',t')$ are pairs of distinct vertices of $G$, then
\[ E_{s , t} - O_{s , t} \equiv    E_{s' , t'} - O_{s' , t'}  \pmod {k+1}.\] 
\end{theorem*}

\section{A block matrix construction.}

Throughout this paper, all graphs are assumed to be connected and loop-free. We allow parallel edges.

\begin{definition} Let $G$ be a graph. Arbitrarily apply directions to the edges in $G$, and  let $M^*$ be the incidence matrix associated with this digraph; columns indexed by edges and rows by vertices. Select a vertex $w$ in $V(G)$, and delete the row indexed by $w$ in $M^*$. Call this new matrix $M$. Let $k$ be a positive integer. Define a \emph{$k$-duplicated signed incidence matrix} (herein \emph{$k$DSI matrix}) of $G$ to be the block matrix $$\left. \left[ \begin{array}{c} M\\ \hline M \\ \hline \vdots \\ \hline M \end{array} \right] \right\}\text{$k$ times}.$$ Further, we call $w$ the \emph{special vertex} in the construction of this $k$DSI matrix. \end{definition}

Our interests lie in graphs $G$ that have $|E(G)| = k(|V(G)|-1)$ for some integer $k$, as this results in a square $k$DSI matrix and allows for permanent calculations.

\begin{definition} Let $A=(a_{i,j})$ be an $n$-by-$n$ matrix. The \emph{permanent} of $A$ is $$\text{Perm}(A) = \sum_{\sigma \in S_n} \prod_{i=1}^n a_{i,\sigma(i)},$$ where the sum is over all elements of the symmetric group $S_n$. \end{definition}

If a particular $\sigma \in S_n$ is such that $ \prod_{i=1}^n a_{i,\sigma(i)} \neq 0$, we will say that it \emph{contributes} to the permanent. We may alternately define a contribution from an appropriate selection of non-zero elements in the matrix.

From the definition of the permanent, we see that it is the determinant with signs not taken into account. In fact, the permanent also can be computed using cofactor expansion, similar to the determinant. As $1 \equiv -1 \pmod{2}$, any square matrix $M$ has $\text{Perm}(M) \equiv \det (M) \pmod{2}$, which suggests that the permanent may have some interesting properties modulo integers.

\begin{rem} \label{rowops} From the definition of the permanent, it is clear that we may interchange two rows or columns without affecting the permanent. Further, multiplying a row or column by a constant results in the permanent being multiplied by that constant. \end{rem}

What happens when a multiple of one row is added to another is less clear, and in general not well behaved. However, there is greater control with the $k$DSI matrix modulo $k+1$, which will be examined in Lemma~\ref{redlem} and Corollary~\ref{reduction}.

\begin{lem}\label{comrow} For an $n \times n$ matrix $M$, if there is a set $\{a_1,a_2,...,a_m\}$ such that rows $r_{a_1}, r_{a_2}, ..., r_{a_m}$ are equal, there is a factor of $m!$ in the permanent of $M$. \end{lem}

\begin{proof}We may write \begin{align*} \text{Perm}(M) &= \sum_{\sigma \in S_n} \prod_{i=1}^n a_{i,\sigma(i)} \\ &=m! \sum_{\sigma \in S_n^*} \prod_{i=1}^n a_{i,\sigma(i)}, \end{align*} where $S_n^*$ is the set of elements of the symmetric group such that $\sigma(a_1)<\sigma(a_2)< \cdots < \sigma(a_m)$, and the $m!$ term allows for further permutations of these elements.   \end{proof}

\begin{lem}\label{redlem} Let $M$ be a matrix and $r_i$ and $r_j$ rows of $M$, $r_i \neq r_j$ as vectors. Suppose there are $k$ copies of $r_j$ in $M$.  Let $M'$ be a matrix derived from $M$ by adding a constant integer multiple of $r_j$ to $r_i$. Then $\text{Perm}(M) \equiv \text{Perm}(M') \pmod {k+1}$. \end{lem}

\begin{proof}   Suppose that $$M = (m_{x,y}) = \left[ \begin{array}{c} r_1 \\ r_2 \\ \vdots  \end{array} \right], \text{ } M' = ({m'}_{x,y}) =  \left[ \begin{array}{c} r_1 \\ \vdots \\ r_i +cr_j  \\ \vdots  \end{array} \right].$$ Define $N$ as the matrix $M$ with row $i$ removed. We will use $N_t$ to denote the matrix $N$ with column $t$ removed. By cofactor expansion along the $i^\text{th}$ row, \begin{align*} \text{Perm}\left(M\right) &= \sum_{t=1}^{kn} m_{i,t}  \text{Perm}\left(N_t\right),\\ \text{Perm}\left(M'\right) &= \sum_{t=1}^{kn} {m'}_{i,t} \text{Perm}(N_t) \\ &=  \sum_{t=1}^{kn}(m_{i,t} + cm_{j,t}) \text{Perm}(N_t) \\ &=  \text{Perm}\left(M\right) + c  \text{Perm} \left[ \begin{array}{c} r_1 \\ \vdots \\ r_{i-1} \\ r_j \\ r_{i+1} \\ \vdots  \end{array} \right]  .\end{align*} As this last matrix has $k+1$ copies of row $r_j$, it has permanent congruent to zero modulo $k+1$ by Lemma~\ref{comrow}.\end{proof}

Throughout this paper, we will consider only matrix operations performed simultaneously in all blocks. The following corollary follows immediately from Lemma \ref{redlem}.

\begin{corollary}\label{reduction} Suppose $M$ is a block matrix made of $k$ identical blocks stacked, and $r_i$ and $r_j$ are rows of $M$ in a common block, $i \neq j$. Let $M'$ be a matrix derived from $M$ by adding a constant integer multiple of $r_j$ to $r_i$ in each block. Then $\text{Perm}(M) \equiv \text{Perm}(M') \pmod {k+1}$. \end{corollary}

\begin{prop} \label{specialchoice} The choice of special vertex only affects the overall sign of the permanent modulo $k+1$  in a $k$DSI matrix. If $k$ is odd, changing the special vertex results in a sign change. If $k$ is even, changing special vertex has no effect on the permanent. \end{prop}

\begin{proof} For signed incidence matrix $M^*$, let $r_1, ... , r_n$ be the rows associated to vertices $1,...,n$ in the original graph $G$, and suppose vertex $i$ is the special vertex, $i \in \{1,...,n\}$. Then, $$r_i = -(r_1 + r_2 + \cdots + r_{i-1} + r_{i+1} + \cdots + r_n),$$ a property of the signed incidence matrix. For all blocks in $M$, we may therefore turn row $r_j$, $i \neq j$, into row $r_i$ using the above equation. By Corollary~\ref{reduction}, only the multiplication of a row in each block by $-1$ affects the permanent modulo $k+1$, flipping the overall sign once for each block. This produces the $k$DSI matrix where $j$ was the special vertex; the permanent is unaffected if there is an even number of blocks, and multiplied by $-1$ if there is an odd number of blocks.\end{proof}

\begin{corollary} \label{oddcase} Any $k$DSI matrix from a graph $G$ with $|V(G)|>2$ and odd $k$ has permanent zero modulo $k+1$. \end{corollary}

\begin{proof} Consider $k$DSI matrix $M$ from graph $G$ where $v_1$ is the special vertex, and suppose $v_2,v_3 \in V(G) \setminus v_1$. Let $M_{v \rightarrow w}$ denote the matrix where the special vertex has been changed from $v$ to $w$, as in Proposition \ref{specialchoice}. If $k$ is odd, then, \begin{align*} \text{Perm}(M) &= -\text{Perm}(M_{v_1 \rightarrow v_2})= \text{Perm}\left( (M_{v_1 \rightarrow v_2})_{v_2 \rightarrow v_3} \right) \\ \text{Perm}(M) &= -\text{Perm}(M_{v_1 \rightarrow v_3}).\end{align*}  As $\left( (M_{v_1 \rightarrow v_2})_{v_2 \rightarrow v_3} \right)$ and $M_{v_1 \rightarrow v_3}$ differ only by interchanged rows, $$\text{Perm}\left( (M_{v_1 \rightarrow v_2})_{v_2 \rightarrow v_3} \right) = \text{Perm}(M_{v_1 \rightarrow v_3}),$$ and it follows that $\text{Perm}(M) = 0$.   \end{proof}

To understand why the restriction in the previous corollary that $|V(G)|>2$ is necessary, we require the following theorem.

\begin{wilson} A number $p$ is prime if and only if $(p-1)! \equiv -1\pmod{p}$. For a composite number $n>4$, $(n-1)! \equiv 0 \pmod{n}$. \end{wilson}

\begin{theorem}\label{specialinvariant} Suppose $G$ is a graph such that $|E(G)| = k(|V(G)|-1)$. With a fixed orientation to the edges, the permanent of the $k$DSI matrix of $G$ is invariant under choice of special vertex modulo $k+1$. \end{theorem}

\begin{proof} By Proposition \ref{specialchoice} and Corollary \ref{oddcase}, the permanent is invariant modulo $k+1$ under choice of special vertex if $k$ is even or $k$ is odd and $|V(G)|>2$. As a graph with only a single vertex creates an empty matrix, it remains to be shown that the permanent is invariant if $|V(G)| = 2$ and $k+1$ is even.

 Suppose then that $G$ is a graph with two vertices and $k$ parallel edges. Applying an arbitrary orientation to the edges, the $k$DSI matrix of this graph has entirely nonzero entries, each column either all $1$ or all $-1$. Thus, it trivially has permanent $\pm k!$. As $k+1$ is even, it follows from Wilson's Theorem that $k! \equiv 0 \pmod{k+1}$ if $k+1 > 4$. Hence, we need only consider $k \in \{1,3\}$. As $\pm 1! \equiv 1 \pmod{2}$ and $\pm 3! \equiv 2 \pmod{4}$, the permanent is invariant under any choice made in constructing the $k$DSI matrix, as desired. \end{proof}

\begin{definition} Let $G$ be a graph with $|E(G)| = k(|V(G)|-1)$ for some integer $k$. Let $M$ be a $k$DSI matrix of $G$. The \emph{graph permanent} of $G$ is $\pm \text{Perm}(M) \pmod {k+1}$, which we will by convention take as a residue in $\left[0,\frac{k+1}{2}\right]$. \end{definition}

The fact that we must choose $\pm \text{Perm}(M)$ is the result of variability in the choice of underlying edge orientation. In changing the direction of an edge in the orientation, a column of the matrix is multiplied by negative one, and hence the permanent changes sign.

\section{Graphic Interpretation}

Consider a graph $G$ with $|E(G)| = k(|V(G)|-1)$ and an associated $k$DSI matrix $M$ with special vertex $v \in V(G)$. For each contribution to the permanent, precisely one non-zero value is selected from each row and similarly from each column. Fix such a contribution. Given the block structure that is used to create matrix $M$, we may associate each block with a unique colour. Then, each edge is selected once, and each non-special vertex $k$ times. Assign colour $c$ to an edge if the contribution uses a value in the associated column that is in the $c^\text{th}$ block. For each coloured edge, assign a tag on the edge close to the vertex that uses that edge in $M$. Such a contribution and colouring scheme on $K_4$ is given in Figure~\ref{markedmatrix}.

\begin{figure}[h]
  \centering
      \includegraphics[scale=.95]{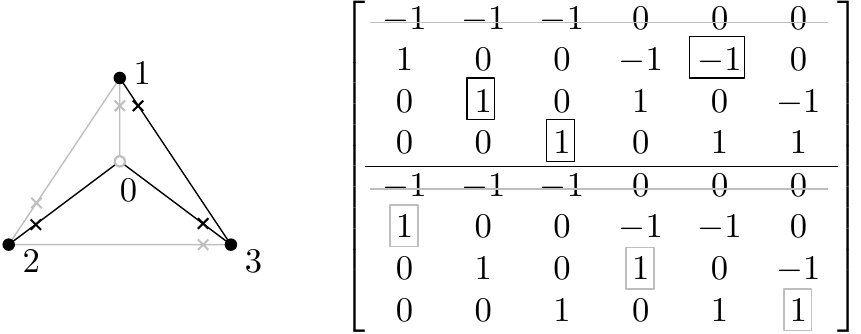}
  \caption{A 2DSI matrix of $K_4$ with special vertex $0$ and a contribution to the permanent.}
\label{markedmatrix}
\end{figure}

Note from this construction that the special vertex cannot receive a tag. All other vertices must receive precisely $k$ tags, one on an edge of each colour. In fact, an arrangement of edge tags and colours on $G$ that assigns each non-special vertex $k$ tags - one on an edge of each colour - and no tags to the special vertex can immediately be turned into a selection of non-zero entries in the $k$DSI matrix. 

\begin{rem} \label{tagcolourbij} There is a bijection between these assignments of tags and colours and the contributions to the permanent. \end{rem}

We will use this bijection in a number of proofs as a way of considering these contributions as a property of the graph itself.

\begin{rem} \label{ops} There are two operations on the tags and colours of the graph that produce other contributions to the permanent. The first is, for any non-special vertex, we may permute the colours of the $k$ edges that have a tag at that vertex. There are $k!$ ways to perform this permutation. The second, at its most intuitive, is that we may switch which vertex receives the tag on every edge in a cycle where all edges have the same colour. Since the colours of edges that have tags at a common vertex are all interchangeable, though, this may be restated as reversing the directions of the tags of a cycle $C$ in $G$ such that, in $C$, each vertex receives one tag. \end{rem}

Suppose then that the graph $G$ has $n$ vertices. From the colour permuting operation, each valid configuration of tags produces $(k!)^{n-1}$ valid colourings. As the tags determine the position in the original matrix that is selected, choice of edge colours does not affect the value of the contribution. As such, $(k!)^{n-1}$ is a factor in the permanent, and the colours do not matter.

\begin{prop} \label{prime} For non-prime $k+1$, the permanent of any square $k$DSI matrix associated to a graph $G$ with $|V(G)|>2$ is zero modulo $k+1$. \end{prop}

\begin{proof} By Corollary~\ref{oddcase}, we may assume that $k+1$ is odd. If $k+1$ is composite and odd it must be greater than four, and by Wilson's Theorem $k! \equiv 0 \pmod{k+1}$. As $k!$ is a factor in the permanent, the result follows. \end{proof}

It is interesting to note, then, that the value of the permanent for a $k$DSI matrix is determined completely by the tag assignments.

\begin{prop} For a $k$DSI matrix, we may produce all contributions to the permanent from a single contribution and the two operations stated in Remark~\ref{ops}. \end{prop}

\begin{proof} By Remark~\ref{ops}, it is sufficient to show that all valid tag placements can be obtained. Starting from a fixed orientation, then, consider getting to another by switching which vertex receives the tag on a set of edges. Since each edge must receive $k$ tags, it is immediate that this selection of edges must be a collection of cycles. \end{proof}

\section{Invariance Under Period Preserving Operations}

We now begin to explore the invariance of the graph permanent under the graph operations that are known to preserve the Feynman period, as mentioned in the introduction. Since we are always restricted to graphs $G$ such that $|E(G)| = k(|V(G)|-1)$ for some integer $k$, we split this section into three subsections, corresponding to specific graph classes that are subsets of this larger class; decompleted $2k$-regular graphs, graphs where $|E(G)| = 2(|V(G)| -1)$, and finally the intersection of the previous two classes, decompleted $4$-regular graphs.

\subsection{Decompleted $2k$-regular graphs}

The following definition is generalized from the introduction.

 \begin{definition}\label{decompdef} For a regular graph $\Gamma$ and any $v \in V(\Gamma)$, the graph $G= \Gamma - v$ is a \emph{decompletion} of $\Gamma$, and $\Gamma$ is the unique \emph{completion} of $G$. \end{definition}

\begin{theorem} \label{bijection} Let $\Gamma$ be a $2k$-regular graph. For vertices $v,w \in V(\Gamma)$, let $M$ be the $k$DSI matrix of $\Gamma - v$ with respect to special vertex $w$, and $N$ the matrix from opposite deletion and special vertex. Then, $\text{Perm}(M) = \pm \text{Perm}(N)$. If $k$ is even and an edge orientation based on an Eulerian circuit in $\Gamma$ is used to construct $M$ and $N$, then $\text{Perm}(M) = \text{Perm}(N)$. \end{theorem}

\begin{proof} By Remarks~\ref{tagcolourbij} and \ref{ops}, it suffices to find a bijection between contributions to the two permanents that is consistent in either maintaining or changing signs. Hence, consider an extension of the taggings from the decompleted graphs to $\Gamma$ by assuming the decompleted vertex received the tags of all edges incident to it. By this extension, the special vertex receives no tags, the decompletion vertex receives $2k$, and all other vertices receive $k$ tags. The bijection that arises naturally, then, is to switch which vertex receives a tag on every edge. By construction, this switches between extensions of contributions to $M$ and $N$.

We now consider the signs of the contributions. Consider a fixed orientation in $\Gamma$ from an Eulerian circuit, and use this to define the underlying orientation for $\Gamma-v$ and $\Gamma-w$. Switching the position of a tag on a single edge changes the sign of an entry of the contribution. Hence, each edge common to both $\Gamma-v$ and $\Gamma-w$ causes a sign change. Suppose now that there are $l$ edges  between $v$ and $w$ in $\Gamma$.  If $l$ is even, then an even number of changes are ignored in deleting $v$ or $w$, and hence this causes no additional sign change. If $l$ is odd, then $w$ in $\Gamma-v$ and $v$ in  $\Gamma-w$ will be incident with opposite parity number of tags that disagree with the underlying orientation, causing an additional sign change. In all cases, it was the structure of the graph, not that actual position of the tags, that determined sign changes between the two permanents. Hence, this bijection either preserves the signs of all contributions or changes all signs. Therefore, $\text{Perm}(M) = \pm \text{Perm}(N)$, as desired. 

If, in addition, $k$ is even, $|E(\Gamma)| = k|V(\Gamma)|$ is even. Again, supposing that edge $\{u,v\}$ occurs $l$ times, there are $k|V(\Gamma)| - 4k +l$ edges not incident with either $v$ or $w$ in $\Gamma$. No matter the parity of $l$, then, an even number of sign changes made are made in the bijection, and overall sign is preserved. \end{proof}

\begin{theorem}\label{invariance} For a fixed $2k$-regular graph $\Gamma$, the graph permanents of all possible decompletions of $\Gamma$ are equal. \end{theorem}

\begin{proof} By Theorem~\ref{specialinvariant}, the choice of special vertex does not affect the permanent up to sign modulo $k+1$. From Theorem~\ref{bijection} and the fact that changing the direction of an edge in the underlying orientation only affect the overall sign of the permanent, the choice of deleted and special vertex may be interchanged without affecting the permanent up to sign. Thus, any two vertices are interchangeable with any other two as the deleted and special vertex, only potentially changing overall sign modulo $k+1$. As the graph permanent from $k$DSI matrix $M$ is $\pm \text{Perm}(M) \pmod{k+1}$, this completes the proof. \end{proof}

The extension of edge taggings from decompleted graph $G$ to $2k$-regular completion $\Gamma$, as seen in Theorem~\ref{bijection} provides the framework for a more graph theoretic look at the taggings. As stated in the introduction, fix a \emph{reference orientation} on the edges of $\Gamma$. We say that an arbitrary orientation is \emph{odd} (\emph{even}) if it disagrees with the reference orientation on an odd (even) number of edges. For distinct vertices $s,t \in V(\Gamma)$, we call any orientation that has $\mathrm{deg}^+(s) = 2k = \mathrm{deg}^-(t)$ and $\mathrm{deg}^+(v) = k = \mathrm{deg}^-(v)$ for all $v \in V(\Gamma) \setminus \{s,t\}$ a \emph{$s$-to-$t$ orientation}. Let $E_{s,t}$ ($O_{s,t}$) denote the number of even (odd) $s$-to-$t$ orientations. The following theorem, mentioned in the introduction, can now be proved.

\begin{theorem}\label{agreementidentity} Let $\Gamma$ be a $2k$-regular graph and $(s,t)$, $(s',t')$ pairs of distinct vertices. Then, $$E_{s,t} - O_{s,t} \equiv E_{s',t'} - O_{s',t'}\pmod{k+1}.$$ \end{theorem}

\begin{proof} Consider first a contribution to the permanent of a $k$DSI matrix $M$ from the decompletion of $\Gamma$ created by deleting vertex $t$ and making $s$ special. With each element selection, we assign a tag that either agrees or disagrees with the reference orientation. If it agrees, we have selected a $1$ from $M$. Otherwise, we have selected a $-1$. As in Theorem~\ref{bijection}, we extend the tagging of our decompleted graph to $2k$-regular graph to create an $s$-to-$t$ orientation where $s$ is the special vertex and $t$ is the decompleted vertex. Assuming an underlying edge orientation from an Eulerian circuit, a $2k$-regular graph will disagree with an $s$-to-$t$ orientation on precisely $k$ edges incident to $t$. Hence $$\text{Perm}(M) = (-1)^k \left( E_{s,t} - O_{s,t}\right).$$

It follows from Theorem~\ref{bijection} that if $k$ is even, $E_{s,t} - O_{s,t} = E_{t,s} - O_{t,s}$. From this construction and Theorem~\ref{specialinvariant}, $E_{s,t} - O_{s,t} \equiv E_{s',t} - O_{s',t} \pmod{k+1}$ for any $s' \in V(\Gamma) \setminus t$. Hence, using an Eulerian circuit as an underlying orientation, $E_{s,t} - O_{s,t} \equiv E_{s',t'} - O_{s',t'}\pmod{k+1}.$

For odd $k$, it was noted in Corollary~\ref{oddcase} and the proof of Theorem~\ref{specialinvariant} that if $|V(\Gamma)| > 3$, or $|V(\Gamma)| = 3$ and $k>3$, the permanent of any decompletion is equal to zero modulo $k+1$, which completes the proof in this case. Finally, by Wilson's Theorem again, if $|V(\Gamma)| = 3$ and $k \in \{1,3\}$, the permanent of any $k$DSI matrix is invariant under sign.

We complete the proof by noting that in turning an Eulerian circuit into any other orientation, each change in the direction of an edge multiplies both sides of the equation by negative one. \end{proof}

\begin{prop}\label{schnetz} Consider two $2k$-regular graphs that differ by a Schnetz twist, seen in Figure~\ref{twist}. Decompletions of these graphs have equal graph permanents. \end{prop}

\begin{figure}[h]
  \centering
      \includegraphics[scale=0.80]{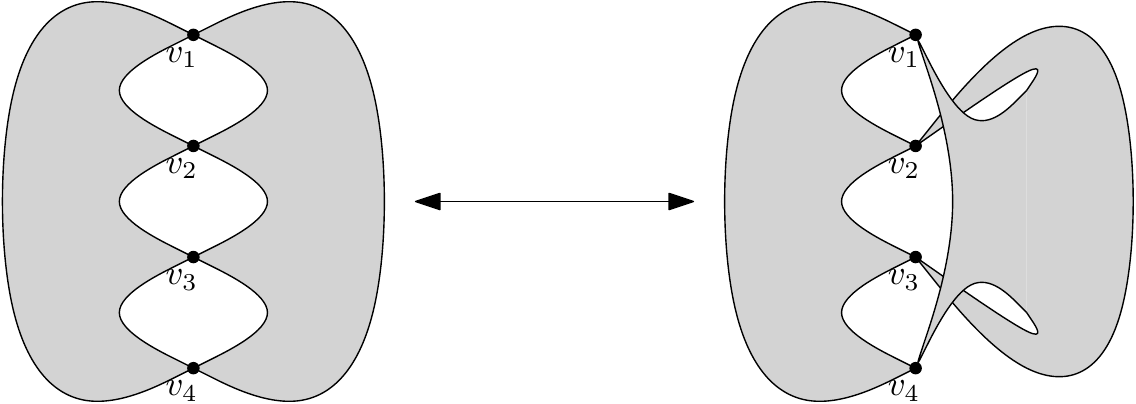}
  \caption{The Schnetz twist.}
\label{twist}
\end{figure}

\begin{proof} As the graph permanent is invariant under choice of special and decompleted vertex, in each graph choose vertex $v_4$ for deletion and $v_3$ as the special vertex. Using Remark~\ref{tagcolourbij} and as in the proof of Theorem~\ref{bijection}, we may capture all contributions to the permanent of the decompleted graph by treating the deleted vertex as a vertex that has received all tags from incident edges, and the special vertex as a vertex that has received none. This will allow for a natural bijection, to be defined below, between contributions after the twisting operation. As both graphs are assumed to be $2k$-regular, it must be the case that vertices $v_3$ and $v_4$ have an equal number of edges in the left side of the graph, say $d_3$, and an equal number of edges in the right side of the graph, which would then be $2k-d_3$. We may say the same for vertices $v_1$ and $v_2$, and denote these $d_1$ and $2k-d_1$.

Suppose then that the left portion of the graph has $n$ vertices contained properly inside. Then, there are $kn + d_1 + d_3$ edges on this side, and hence an equal number of tags. If vertex $v_1$ receives $m$ tags, then vertex $v_2$ must receive $$(kn + d_1 + d_3) - (kn + d_3 + m) = d_1 - m,$$ accounting for all tags on this side. Similarly, since all vertices receive $k$ tags, $v_1$ must receive $k-m$ tags on the right, and $v_2$ must receive $k-d_1 + m$.

After the twist, then, consider switching the sides of the tags on the right side only. Again, the deleted vertex receives $2k$ tags and the special vertex receives none. Further, vertex $v_1$ receives $m + ((2k-d_1) -(k+m-d_1)) = k$ tags, and similarly $v_2$ receives $(d_1-m)+((2k-d_1) - (k-m)) = k$ tags. Hence, this is a contribution to the permanent. This operation is clearly bijective and at most changes the overall sign of all contributions, and hence the graph permanents must be equal. \end{proof}

\subsection{Graphs $G$ where $|E(G)| = 2(|V(G)| -1)$}

\begin{prop}\label{dual} For a graph $G$ where $|E(G)| = 2(|V(G)| -1)$ and planar dual $G^*$, the graph permanents for $G$ and $G^*$ are equal. \end{prop}

\begin{proof} By Remark~\ref{rowops} and Corollary~\ref{reduction} we may row perform row operations without affecting the graph permanent. After row reduction, we may write the 2DSI matrix $$M_G = \left[ \begin{array}{c|c} I & A \\ \hline I & A \end{array} \right].$$ The permanent of this matrix is $2^{|A|}\text{Perm}(A)$. Dually, we have 2DSI matrix $$M_{G^*} = \left[ \begin{array}{c|c} -A^t & I \\ \hline -A^t & I \end{array} \right],$$ with permanent $(-2)^{|A|} \text{Perm} (A)$.\end{proof}

\begin{lem}\label{pigeon} If a square block matrix has a non-square block that is the only block containing non-zero entries in its particular row and column set, the permanent is zero.  \end{lem}

\begin{proof} This follows immediately from the pigeonhole principle and the definition of the permanent. \end{proof}

\begin{theorem} \label{2vertexcut} Consider the graph $G$ and two minors $G_1$ and $G_2$ seen in Figure~\ref{2cut}. If for $G^* \in \{G,G_1,G_2\}$, $2|V(G^*)| = |E(G^*)|-2$, then the graph permanent of $G$ is equal to the product of the graph permanents of $G_1$ and $G_2$. \end{theorem}

\begin{figure}[h]
  \centering
      \includegraphics[scale=1.20]{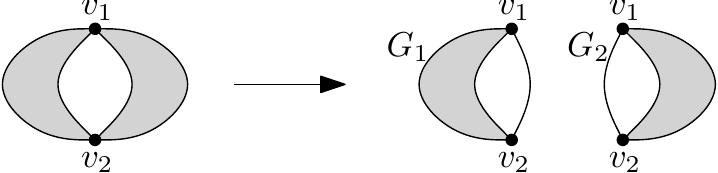}
  \caption{Operation on a two vertex cut.}
\label{2cut}
\end{figure}

\begin{proof} Let $G$ be the original graph, and $G_1$ and $G_2$ the minors as seen in Figure~\ref{2cut}. Maintain a constant edge orientation across all graphs, with the edge $\{v_1,v_2\}$ oriented towards $v_2$. Choosing $v_2$ as the special vertex for all graphs, the respective $2$DSI matrices are
 $$M_G = \left[ \begin{array}{ccc|ccc}
& \bf{\bf{G_1}} & & & \bf{0}& \\ 
&C& && D& \\ 
 & \bf{0} & & & \bf{\bf{G_2}} & \\ \hline
& \bf{\bf{G_1}} & & & \bf{0}& \\  
&C& &&D& \\ 
 & \bf{0} & & & \bf{\bf{G_2}} & \\  \end{array} \right], $$  

 $$M_{G_1} = \left[ \begin{array}{ccc|c}
& \bf{G_1}& & \bf{0} \\  
&C& & 1 \\ \hline
& \bf{\bf{G_1}} & & \bf{0} \\ 
&C& & 1
 \end{array} \right], $$ and 

 $$M_{G_2} = \left[ \begin{array}{c|ccc}
\bf{0} && \bf{\bf{G_2}} &  \\ 
1 &&D&  \\ \hline
\bf{0} && \bf{\bf{G_2}} &  \\  
1&&D& 
 \end{array} \right], $$ where $(C|D)$ is the row corresponding to vertex $v_1$.

Using cofactor expansion along the last column, we get $$\text{Perm}(M_{G_1}) =  \text{Perm}\left[ \begin{array}{c} \bf{G_1} \\  \bf{G_1} \\  C \end{array} \right] + \text{Perm}\left[ \begin{array}{c} \bf{G_1} \\  C \\  \bf{G_1} \end{array} \right] =  2  \text{Perm}\left[ \begin{array}{c} \bf{G_1} \\  \bf{G_1} \\  C \end{array} \right] .$$ Similarly, $$ \text{Perm}(M_{G_2}) = 2  \text{Perm} \left[ \begin{array}{c} \bf{G_2} \\ \bf{G_2} \\ D  \end{array}  \right].$$  Again, use $N_r$ to denote matrix $N$ with column $r$ deleted. Letting $C = (c_1,c_2,...)$ and $D=(d_1,d_2,...)$, we use cofactor expansion along the bottom row to get, \begin{align*}  \text{Perm}(M_{G_1}) &= 2 \sum_i c_i  \text{Perm} \left[ \begin{array}{c} \bf{G_1} \\ \bf{G_1} \end{array} \right]_i, \\ \text{Perm}(M_{G_2}) &= 2 \sum_i d_i  \text{Perm} \left[ \begin{array}{c} \bf{G_2} \\ \bf{G_2} \end{array} \right]_i. \end{align*}

We now compute $\text{Perm}(M_G)$ by cofactor expansion along the two rows $(C|D)$. Suppose the block $G_1$ contains $n$ columns and block $G_2$ contains $m$ columns. Blocks will be square only if in the expansion we have deleted one column in the first $n$ columns and one in the last $m$. By Lemma~\ref{pigeon}, then, \begin{align*} \text{Perm}(M_G) &= 2 \sum_{\substack{1 \leq i \leq n \\ n < j \leq n+m}} c_i d_j \text{Perm} \left[ \begin{array}{c|c} \begin{array}{c} \bf{G_1} \\ \bf{G_1} \end{array} & \bf{0}\\ \hline \bf{0} & \begin{array}{c} \bf{G_2} \\ \bf{G_2} \end{array} \end{array} \right]_{i,j} \\ &= 2 \sum_{\substack{ 1\leq i \leq n \\ 1 \leq j \leq m}} c_i d_j \text{Perm} \left[ \begin{array}{c} \bf{G_1} \\ \bf{G_1} \end{array} \right]_i   \text{Perm} \left[ \begin{array}{c} \bf{G_2} \\ \bf{G_2} \end{array} \right]_j \\ & \equiv - \text{Perm}(M_{G_1}) \text{Perm}(M_{G_2}) \pmod{3}.\end{align*} Hence, the graph permanent of $G$ is equal to the product of graph permanents of $G_1$ and $G_2$.  
\end{proof}

\subsection{Decompleted $4$-regular graphs}

\begin{corollary}\label{3cut}[to Theorem~\ref{2vertexcut}] With 4-regular graphs $\Gamma$, $\Gamma_1$, and $\Gamma_2$ as in Figure~\ref{3vertexcut}, the graph permanents of any decompletion of $\Gamma$ is equal to the product of graph permanents of decompletions of $\Gamma_1$ and $\Gamma_2$. \end{corollary}

\begin{figure}[h]
  \centering
      \includegraphics[scale=1.10]{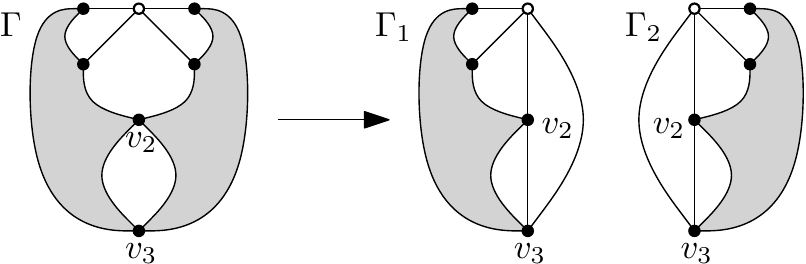}
  \caption{The completed graph with three vertex cut, corresponding to the completion of the graphs in Figure~\ref{2cut}.}
\label{3vertexcut}
\end{figure}

\begin{proof} By Theorem~\ref{invariance}, graph permanent of a decompleted 4-regular graphs is invariant under choice of decompletion vertex. Hence, decompleting each at the vertex labeled with a hollow circle, we produce graphs as in Theorem~\ref{2vertexcut}, and the result follows.  \end{proof}

\begin{theorem}\label{4edgecut} Let graphs $\Gamma$, $\Gamma_1$, and $\Gamma_2$ in Figure~\ref{4edgecutpic} be 4-regular. Construct $2$DSI matrices by deleting the vertices labeled with a hollow circle and making the vertices labeled with a square the special vertex, all distinct from vertices incident with the edges in the 4-edge cut. The permanent of the $2$DSI matrix associated to $\Gamma$ is equal to the product of the permanents associated to $2$DSI matrices of $\Gamma_1$ and $\Gamma_2$.  \end{theorem}

\begin{figure}[h]
  \centering
      \includegraphics[scale=1.30]{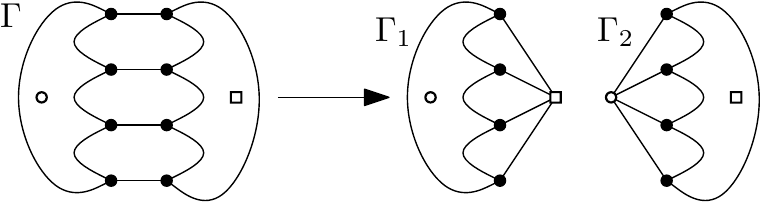}
  \caption{The graph permanent reduction for a graph with a 4-edge cut.}
\label{4edgecutpic}
\end{figure}

\begin{proof} Graph $\Gamma$ has associated $2$DSI matrix $$M_{G} = \left[ \begin{array}{c|cc}
I_4 & A & \bf{0} \\
\bf{0} & B & \bf{0} \\
-I_4 & \bf{0} & C \\
\bf{0} & \bf{0} & D \\ \hline
I_4 & A & \bf{0} \\
\bf{0} & B & \bf{0} \\
-I_4 & \bf{0} & C \\
\bf{0} & \bf{0} & D \end{array} \right], $$ where blocks $A$ and $B$ correspond to the left side of the graph, and blocks $C$ and $D$ correspond to the right. Supposing there are $l$ edges in the left subgraph and $k$ edges in the right, then block $B$ consists of $\frac{l-4}{2}$ rows and block $D$ has $\frac{k-8}{2}$ rows.

Consider cofactor expansion along the first four columns. By Lemma~\ref{pigeon} only matrix minors where the deleted row meets $A$ produce a non-zero permanent. Given that each column has two potential values, and since we may interchange rows freely, $$\text{Perm}(M_{G}) = 2^4 \cdot \text{Perm} \left[ \begin{array}{c|c}
A & \bf{0} \\
 B & \bf{0} \\
\bf{0} & C \\
\bf{0} & D \\ \hline
B & \bf{0} \\
\bf{0} & C \\\bf{0} & D \end{array} \right] = 16 \cdot \text{Perm} \left[ \begin{array}{c} B\\A\\B \end{array} \right] \cdot \text{Perm} \left[ \begin{array}{c} C\\D\\C\\D \end{array} \right].$$  Notice that $$\text{Perm}(M_{G_1}) = \text{Perm} \left[ \begin{array}{c|c} I_4 & A \\ \bf{0} & B \\ \hline I_4 & A \\ \bf{0} & B \end{array} \right] = 16 \cdot \text{Perm} \left[ \begin{array}{c} B\\A\\B \end{array} \right] ,$$ and $$\text{Perm}(M_{G_2}) = \text{Perm} \left[ \begin{array}{c}C\\D\\C\\D \end{array} \right].$$ This completes the proof.  \end{proof}

\subsection{Relation to Feynman periods}
\label{relationtofeynman}

Let $\Gamma$ be a 4-regular graph and let $G=\Gamma - v$ for some $v\in V(\Gamma)$. 
Once again we are thinking of $G$ as a Feynman diagram, specifically as a 4-point graph in scalar $\phi^4$ theory.  The vertices where $G$ used to have edges connecting to $v$ are the \emph{external edges} of $G$.  If we think of $G$ representing some particle interactions, then the external edges represent the four particles coming into or out of the process.
  
To each edge $e$ of $G$ associate a variable $a_e$ and define the \emph{period} of $G$ to be $$\int_{a_i \geq 0} \frac{\Omega}{\Psi^2}, $$ where 
$$\Omega = \sum_{i=1}^{|E(G)|}(-1)^i \prod_{\substack{j=1 \\ j \neq i}}^{|E(G)|}da_j, \hspace{6mm}
\Psi = \sum_{\substack{T \\\text{spanning} \\ \text{tree of }G}}\prod_{e\not\in E(T)}a_e.$$

Provided $K$ is internally 6-edge-connected the period integral converges \cite{bek}.  If $K$ has an internal 4- or 2-edge-cut then we say $G$ has a \emph{subdivergence}.  If $G$ has no subdivergences then we say it is \emph{primitive}.  In quantum field theory, graphs with subdivergences are more complicated because the subdivergences need to be dealt with first in the renormalization process (in this language see \cite{BrKr}).  The period of a primitive graph is renormalization-scheme independent, but is still an informative part of the full Feynman integral. 

Periods of 4-point graphs in $\phi^4$ are preserved by the three operations we have been looking at throughout this paper, namely planar dual, completion followed by decompletion, and Schnetz twist \cite{bkphi4, Sphi4}.  These identities of periods explain all known cases where two primitive 4-point graphs in $\phi^4$ have the same period (see \cite{Sphi4}).  Furthermore if $\Gamma$ has a 3-vertex cut then the period of any decompletion of $\Gamma$ is equal to the product of the periods of graphs $G_1$ and $G_2$ as seen in Figure~\ref{2cut}. Thus, having found a graph invariant which is preserved by the three operations and with the appropriate product property, namely the graph permanent, the following conjecture is natural.

\begin{conjecture} Suppose $G_1$ and $G_2$ are two primitive 4-point graphs in scalar $\phi^4$ theory. If $G_1$ and $G_2$ have equal periods, then they have equal graph permanents. \end{conjecture}

Now consider the case of subdivergences.  Internal 2-edge cuts in $\Gamma$ are not very interesting in this context because they automatically yield small vertex cuts as well.  Suppose $\Gamma$ has an internal 4-edge cut, as in Figure~\ref{4edgecutpic}.  Decompleting to $G$, say $G_2$ is the side which did not have the decompletion vertex.  Then $G_2$ is the subdivergence inside $G$.  The leading term in the renormalized period of $G$ is the product of the periods of $G/G_2$ and $G_2$, which is what is given by the graph permanent.  This further strengthens the suggestion that the graph permanent is measuring something about the period, and gives some initial hints of what kind of thing it could be measuring. 

In comparison, there is one two-valued invariant which we have some handle on, namely whether or not the period is of full transcendental weight (see \cite{BrY, Sphi4}).  This invariant has a different behaviour on subdivergences, namely all subdivergences give weight drop regardless of the weight of their pieces.  This shows that the graph permanent is capturing something different about the period.  We can also compare these two invariants on specific graphs.  For example, using names from \cite{Sphi4}, $P_{6,1}$ and $P_{6,4}$ both have graph permanent of $0$ while $P_{6,1}$ is full weight and $P_{6,4}$ has weight drop. Alternately, $P_{6,2}$ and $P_{6,3}$ both have nonzero graph permanent while $P_{6,2}$ is full weight and $P_{6,3}$ has weight drop.  Note that the names here label the completed graphs, or alternately label families of 4-point graphs related by completion.

The $c_2$ invariant is conjecturally another graph invariant with the same symmetries of the period \cite{BrS}.  It does specialize to capture the question of weight drop and, relatedly, is $0$ on graphs with subdivergences.  Thus the graph permanent is again capturing something different about the period.

Unfortunately, it simply is not clear what about the period the graph permanent is measuring.  The graph permanents for the completion-families of graphs are given in Appendix~\ref{calculatedperms}.

A final hint at the existence and potential nature of a connection between the graph permanent and the Feynman period comes from the fact that both of them are closely related to questions of (momentum) flows in graphs.  This is explained in the next section after the required graph theory of flows is presented.

\section{A connection to nowhere-zero flows}
\label{flow_sec}

The graph permanent also arises naturally in the study of nowhere-zero flows on graphs, and in this section we will develop this interesting connection.  To  this end we demonstrate that a graph has a certain orientation (closely related to the study of flows) under the assumption that a spanning subgraph has nonzero graph permanent.

Suppose that $G$ is a graph and (as before) direct the edges of $G$ arbitrarily.  For an abelian group $\mathcal{G}$, a function $\phi : E(G) \rightarrow \mathcal{G}$ has an associated boundary function $\partial \phi : V(G) \rightarrow \mathcal{G}$ given by the following rule (here $\delta^+(v)$ denotes the edges directed away from $v$ and $\delta^-(v)$ the edges directed to $v$);
\[ \partial \phi (v) = \sum_{ e \in \delta^+(v) } \phi(e) - \sum_{ e \in \delta^-(v) } \phi(e). \]
Note that $\sum_{v \in V(G)} \partial \phi (v) = 0$ since each edge $e$ contributes $\phi(e) - \phi(e) = 0$ to the total sum.  We define $\phi$ to be a $\mathcal{G}$-\emph{flow} if $\partial \phi$ is the zero function, and we say that $\phi$ is \emph{nowhere-zero} if $0 \not\in \phi ( E(G) )$, where $\phi(E(G))$ is the \emph{range} of $\phi$.  Note that if $\phi$ is a nowhere-zero flow, we may alter the orientation of $G$ by reversing the direction of some edge $e$ and modify $\phi$ by replacing $\phi(e)$ with its negation, and this results in another nowhere-zero flow.  Therefore, the question of whether our graph has a nowhere-zero flow in a particular group $\mathcal{G}$ will be independent of the chosen orientation.  Accordingly, we will say that an undirected graph has a nowhere-zero flow if some (and thus every) orientation permits such a map.   Nowhere-zero flows have a rich history initiated by Tutte who proved all of the following properties.  (Throughout we let $\mathbb{Z}_k = \mathbb{Z}/ k \mathbb{Z}$).  

\begin{theorem}[Tutte \cite{Tu49} \cite{Tu54}] $\mbox{}$ 
\begin{enumerate}
\item If $G$ and $G^*$ are dual planar graphs, then $G$ has a $k$-colouring if and only if $G^*$ has a nowhere-zero $\mathbb{Z}_k$-flow.

\item A graph $G$ has a nowhere-zero $\mathbb{Z}_k$-flow if and only if it has a $\mathbb{Z}$-flow with range a subset of $\{ \pm 1, \pm 2, \ldots, \pm (k-1) \}$.

\item If $G$ has a nowhere-zero $\mathcal{G}$-flow for a finite abelian group $\mathcal{G}$, then it has a nowhere-zero $\mathcal{G}'$-flow for every abelian group $\mathcal{G}'$ with $|\mathcal{G}'| \ge |\mathcal{G}|$.  
\end{enumerate}
\end{theorem}

In addition, Tutte made three famous conjectures which have motivated a tremendous amount of investigation, but all remain unsolved.  

\begin{conjecture}[Tutte] $\mbox{}$
\begin{enumerate}
\item (5-Flow) Every graph without a cut-edge has a nowhere-zero $\mathbb{Z}_5$-flow.
\item (4-Flow) Every graph without a cut-edge and without a Petersen minor has a nowhere-zero $\mathbb{Z}_4$-flow.
\item (3-Flow) Every 4-edge-connected graph has a nowhere-zero $\mathbb{Z}_3$-flow.
\end{enumerate}
\end{conjecture}

In his investigations of flows, Jaeger introduced the following interesting concept.  For an (undirected) graph $G$, define an orientation of its edges to be a \emph{modulo} $k$ orientation if every vertex $v$ satisfies $|\delta^+(v)| - |\delta^-(v)| \equiv 0 \pmod{k}$.  If we have a modulo $k$ orientation of a graph, then we can obtain a nowhere-zero $\mathbb{Z}_k$ flow by assigning each edge to have flow value $1$.  On the other hand, if we have found a flow $\phi: E(G) \rightarrow \mathbb{Z}_k$ for which $\phi ( E(G) ) \subseteq \{ -1, 1\}$ then by reversing the edges with flow value $-1$, we obtain a modulo $k$ orientation.  So, in short, a modulo $k$ orientation is equivalent to the existence of a $\mathbb{Z}_k$-flow with range a subset of $\{-1,1\}$.  Jaeger offered the following unifying conjecture which, if true, is known to imply both Tutte's 5-flow and 3-flow conjectures.

\begin{conjecture}[Jaeger \cite{Ja84}]
Every $4k$-edge-connected graph has a modulo $2k+1$ orientation.  
\end{conjecture}

A recent flurry of activity started by Thomassen's proof of a weak version of the 3-flow conjecture \cite{Th} has resulted in a proof of a weak version of the above conjecture.  Namely, Lov\'{a}sz, Thomassen, Wu, and Zhang \cite{LoThWuZh} have recently proved that every $(3k-3)$-edge-connected graph has a modulo $k$ orientation.  

With an eye toward constructing flows and modulo $k$ orientations, let us now return to incidence matrices.  As before, we shall use $M^*$ to denote the oriented incidence matrix of $G$.  If we regard $M^*$ as a matrix with entries in $\mathcal{G}$, then a flow in $\mathcal{G}$ is precisely a vector in the nullspace of $M^*$.  Since the sum of the rows of $M^*$ is zero, the matrix $M^*$ will have the same nullspace as the matrix $M$ which we obtain from $M^*$ by deleting the row corresponding to an arbitrarily chosen special vertex $w$.  This recasts the problem of finding a nowhere-zero $\mathcal{G}$ flow in $G$ as one of finding a vector in the nullspace of $M$ with no zero entries.  Turning our attention to the special case $\mathcal{G} = \mathbb{Z}_k$, we note that the existence of a modulo $k$ orientation in $G$ is equivalent to the existence of a $\pm 1$ valued vector in the nullspace of $M$.  

The main tool we will require for our result is the following ``polynomial method''.

\begin{theorem}[Alon and Tarsi \cite{AlTar}]
\label{polymeth} 
Let $\mathbb{F}$ be a field and let $f(x_1, \ldots, x_n)$ be a polynomial in $\mathbb{F}[x_1, \ldots, x_n]$.  Suppose that the coefficient of $x_1^{d_1} x_2^{d_2} \ldots x_n^{d_n}$ is nonzero and that $\deg(f) = d_1 + d_2 + \ldots + d_n$.  Then for every $S_1, \ldots, S_n \subseteq \mathbb{F}$ with $|S_i| > d_i$ for $1 \le i \le n$, there exist $s_i \in S_i$ for $1 \le i \le n$ so that $f(s_1, \ldots, s_n) \neq 0$.
\end{theorem}

Now we shall restrict our attention to the case when $\mathcal{G} = \mathbb{Z}_p$ for a prime $p$ so that our matrix $M$ has its entries in a field.  In this setting, the graph permanent arises naturally in conjunction with the above theorem to produce a certificate for a graph which guarantees the existence of a modulo $p$ orientation of $G$.  Namely, we will prove the following result which appears implicitly in \cite{AlLiMe}.

\begin{theorem}[Alon, Linial, Meshulam]
Let $p$ be prime, let $G$ be an $n$ vertex graph, and let $H$ be a spanning subgraph of $G$ with $|E(H)| = (p-1)(n-1)$.  If a $(p-1)$DSI matrix for $H$ has nonzero permanent modulo $p$, then $G$ has a modulo $p$ orientation.
\end{theorem}

\begin{proof}  Begin by orienting the edges of $G$ arbitrarily.  Define $H' = G - E(H)$ and choose $\phi' : E( H' ) \rightarrow \{ -1, 1 \}$ arbitrarily.  Our goal will be to use the polynomial method to prove that $\phi'$ may be extended to a $\mathbb{Z}_p$ flow of $G$ with range a subset of $\{-1,1\}$.  To do so, define $V' = V \setminus \{w\}$ and for every $v \in V'$ define $A_v = \mathbb{Z}_p \setminus \{\partial \phi'(v) \}$.  Now we shall construct a polynomial $f$ with a variable $x_e$ for every edge $e \in E(H)$ by the following rule: 
\[ f = \prod_{v \in V' } \prod_{ a \in A_v } \left( -a + \sum_{e \in \delta_H^+(v)} x_e - \sum_{e \in \delta_H^-(v)} x_e  \right). \]

Let us pause to consider what it would mean for this polynomial to be nonzero on a particular assignment to the variables.  Namely, let $\phi : E(H) \rightarrow \mathbb{Z}_p$ and suppose that evaluating $f$ where each variable $x_e$ is assigned the value $\phi(e)$ gives a nonzero value.  Considering the innermost product in our equation, we see that in order for $f$ to be nonzero when evaluated at $\phi$ it must be that $\partial \phi (v)$ is not equal to $a$ for every $a \in  A_v$.  However, this is precisely equivalent to the statement that $\partial \phi(v) = - \partial \phi'(v)$.  Since this must hold at every $v \in V'$ we have that the function $\phi \cup \phi'$ (i.e. the function which maps each $e \in E(H')$ to $\phi'(e)$ and each $e \in E(H)$ to $\phi(e)$) is a flow.  Indeed, this polynomial evaluated at $\phi$ will result in a nonzero value precisely when $\phi \cup \phi'$ is a flow.  

Now  consider the coefficient of $\prod_{e \in E(H)} x_e$ in the expansion of $f$.  Since this term has degree $|E(H)| = (p-1)(n-1)$ we can see that this is the same as the coefficient of the same term in the expansion of the polynomial 
\[ \prod_{v \in V' } \left( \sum_{e \in \delta_H^+(v)} x_e - \sum_{e \in \delta_H^-(v)} x_e  \right)^{p-1}. \]
However, this is precisely the permanent of the $(p-1)$DSI matrix of $H$ with special vertex $w$.  So, by assumption, this coefficient is nonzero, and then (since $\deg(f) \le (p-1)(n-1)$) by the Alon-Tarsi theorem, we may choose an assignment to these variables $\phi : E(H) \rightarrow \{ -1, 1 \}$ in such a way that evaluating $f$ on these inputs is nonzero.  As we have seen, this gives us a $\mathbb{Z}_p$-flow $\phi \cup \phi'$ with range a subset of $\{-1,1\}$, or equivalently a modulo $p$ orientation of $G$, as desired.
\end{proof}

Further evidence of the connection between our invariant and the Feynman integral can be found in the polynomial appearing in the previous proof.  Indeed, just as the period of a graph can be computed using a number of different integrals (one of which features integrating over a basis of the cycle space), so the polynomial coefficient in our proof has many essentially equivalent variations.  To give a concrete instance of this, let $G = (V,E)$ be a graph with $|E| = 2 |V| - 2$ and assume that $G$ is equipped with an arbitrary orientation of the edges.  Following the line of the previous proof, we will construct a polynomial and use the polynomial technique to show the existence of a nowhere-zero $\mathbb{Z}_3$-flow on $G$ (under certain additional assumptions).  Though the polynomial we construct will be quite different from that in the proof of the previous theorem, it will turn out that the coefficient of interest will be same (up to sign).  

Choose a spanning tree $T$ of $G$ and introduce a variable (from $\mathbb{Z}_3$) denoted $y_e$ for every edge $e \in E \setminus E(T)$.  These variables may be viewed as indexing the $\mathbb{Z}_3$-cycle space of $G$.  For $f \in E(T)$ there is a unique edge-cut in $G$ which contains $f$ but no other edge of $T$, called the fundamental cut of $f$.  Define $C_f^+$ to be the set of edges in this cut (other than $f$) which are oriented the same as $f$ relative to this cut, and $C_f^-$ to be those edges in the cut with opposite orientation to $f$.  Using these, we define the following linear polynomial (with variables $\{ y_e \}_{e \in E \setminus E(T)}$)
\[ g_f = \sum_{e \in C_f^-} y_e - \sum_{e \in C_f^+} y_e. \]
It follows from basic theory that every assignment to all of the $y_e$ variables extends uniquely to a $\mathbb{Z}_3$-flow of $G$.  Furthermore, for this flow, the value on an edge $f \in E(T)$ is given by $g_f$.  

Now we will proceed in our attempt to use the polynomial technique to find a nowhere-zero $\mathbb{Z}_3$-flow in $G$.  To do this, define the polynomial
\[ g = \prod_{f \in E(T)} g_f. \]
If the coefficient of the monomial $\prod_{e \in E \setminus E(T)} y_e$ in the expansion of $g$ is nonzero, then by Theorem \ref{polymeth}, there exists an assignment to the variables using only the elements $\{-1,1\}$ so that $g$ evaluates to a nonzero number.  However, this is precisely what is required to have a nowhere-zero flow.  So, in short, we have now found another polynomial coefficient which, if nonzero, implies the existence of a nowhere-zero $\mathbb{Z}_3$-flow in our graph.  

Although the polynomial $g$ we have constructed depends on the choice of spanning tree, the coefficient of the term $\prod_{e \in E \setminus E(T)} y_e$ (up to sign) does not depend on this choice.  Furthermore, (up to sign) this coefficient is the same as the coefficient of $\prod_{e \in E(H)} x_e$ in the expansion of $f$ from the previous proof.  To see this, let $M$ be a matrix obtained from the incidence matrix of $G$ by removing a row, and assume (for convenience) that the columns of $M$ associated with edges in $T$ appear before those associated with edges in $E \setminus E(T)$.  It follows from basic theory that the matrix $M$ may be turned by row operations  into a matrix for which the columns associated with edges in $T$ form an identity matrix.  So, by row operations we may transform $M$ into a matrix of the form $\begin{bmatrix} I & A \end{bmatrix}$.  Now, working in the field $\mathbb{Z}_3$ we have
\[ \text{Perm} \begin{bmatrix} M \\ M \end{bmatrix} = \pm \text{Perm} \begin{bmatrix} I & A \\ I & A \end{bmatrix} = \pm \text{Perm} A. \]
By the previous proof, the coefficient of $\prod_{e \in E(H)} x_e$ in the expansion of $f$ is $\text{Perm} \begin{bmatrix} M \\ M \end{bmatrix}$, and by elementary reasoning, the coefficient of $\prod_{ e \in E \setminus E(T) } y_e $ in our polynomial $g$ is equal to $\text{Perm} A$, thus yielding the desired connection.

This polynomial $g$ is closely related to the Feynman integrand in momentum space. As before we are taking the simplest possible case of a Euclidean massless scalar field theory.  To set up the Feynman integrand in momentum space first take a basis of the cycle space and assign a variable to each cycle in the basis; in particular the $y_e$ are appropriate.  These variables represent the momentum flowing around the cycles and we view them as taking values in $\mathbb{R}^4$.  To each edge associate the signed sum of the variables for the cycles running through that edge; this is the momentum flowing through that edge and is $g_f$ for $f\in T$ and is $y_e$ itself for $e \not\in T$.  Let 
\[
\tilde{g} = \prod_{f\in T}|g_f|^2\prod_{e\not\in T}|y_e|^2
\]  
where the norms are the usual Euclidean norm.
The Feynman integrand is $1/\tilde{g}$ and the integral runs over all values of the $y_e$.  After modifying the integral to use the projective volume measure, $\sum(-1)^i \prod_{\substack{j \neq i}}dy_j$ analogously to subsection~\ref{relationtofeynman}, this calculates the same period (see \cite{Sphi4} for discussion and proof in close to this language).  

Returning to the polynomials, another way to look at the construction of $\tilde{g}$ is to assign a momentum variable to each edge but then impose momentum conservation at each vertex, which is just different language for the flow condition.  Furthermore, we see that $\tilde{g}$ and $g$ only differ in two ways.  First the norm squared has replaced the simple appearance of variables which is a natural adjustment to vector valued variables.  Second $\tilde{g}$ has a factor for the edges not in $T$.  If we put an analogous factor in $g$ we would be multiplying $g$ by the product of all the variables.  This has essentially no impact on the use of the polynomial technique as it simply shifts up all degrees.  

Ultimately, this close relationship between flow calculations in graph theory and Feynman integrals should not be surprising since the momentum space Feynman integral is the integral over all possible momentum flows through the graph.

\section*{Acknowledgment}

The authors acknowledge the support of NSERC.

\appendix
\section{Primitive $\phi^4$ graphs and their graph permanent} \label{calculatedperms}

\begin{tabular}{|c| p{5.2cm} |p{5.2cm}|}\hline 
$|V|$ & $0$ & $\pm 1$  \\ \hline
5 &  
$ P_{ 3 ,  1 }$ & 
\\ \hline
6 & 
& 
$ P_{ 4 ,  1 }$\\ \hline
7 & 
& 
$ P_{ 5 ,  1 } $\\ \hline
8 & 
$ P_{ 6 ,  1 },$  $P_{ 6 ,  4 } $  & 
$ P_{ 6 ,  2 }, $ $ P_{ 6 ,  3 } $ \\ \hline
9 & 
$ P_{ 7 ,  3 },$ $P_{ 7 ,  5 },$ $P_{ 7 ,  9 },$ $P_{ 7 , 1 0 },$ $P_{ 7 , 1 1 }$ & 
$ P_{ 7 ,  1 }, $ $ P_{ 7 ,  2 }, $ $ P_{ 7 ,  4 }, $ $ P_{ 7 ,  6 }, $ $ P_{ 7 ,  7 }, $ $ P_{ 7 ,  8 } $ \\ \hline
10 & 
$ P_{ 8 ,  5 }, $ $ P_{ 8 ,  6 }, $  $ P_{ 8 ,  9 }, $ $ P_{ 8 , 1 4 }, $ $ P_{ 8 , 1 7 }, $ $ P_{ 8 , 1 8 }, $ $ P_{ 8 , 2 3 }, $ $ P_{ 8 , 2 5 }, $ $ P_{ 8 , 3 1 }, $ $ P_{ 8 , 3 3 }, $ $ P_{ 8 , 3 5 }, $           $ P_{ 8 , 3 9 }, $ $ P_{ 8 , 4 1 }$ & 
$ P_{ 8 ,  1 }, $ $ P_{ 8 ,  2 }, $ $ P_{ 8 ,  3 }, $ $ P_{ 8 ,  4 }, $ $ P_{ 8 ,  7 }, $ $ P_{ 8 ,  8 }, $ $ P_{ 8 , 1 0 }, $ $ P_{ 8 , 1 1 }, $ $ P_{ 8 , 1 2 }, $ $ P_{ 8 , 1 3 }, $ $ P_{ 8 , 1 5 }, $ $ P_{ 8 , 1 6 }, $ $ P_{ 8 , 1 9 }, $ $ P_{ 8 , 2 0 }, $  $ P_{ 8 , 2 1 }, $  $ P_{ 8 , 2 2 }, $ $ P_{ 8 , 2 4 }, $ $ P_{ 8 , 2 6 }, $  $ P_{ 8 , 2 7 }, $ $ P_{ 8 , 2 8 }, $  $ P_{ 8 , 2 9 }, $ $ P_{ 8 , 3 0 }, $ $ P_{ 8 , 3 2 }, $ $ P_{ 8 , 3 4 }, $ $ P_{ 8 , 3 6 }, $  $ P_{ 8 , 3 7 }, $  $ P_{ 8 , 3 8 }, $  $ P_{ 8 , 4 0 } $  \\ \hline
\end{tabular}

\vspace{0.5cm}
The notation used comes from \cite{Sphi4}.

\bibliography{cites}{}
\bibliographystyle{plain}

\end{document}